\documentclass[12pt,reqno]{amsart}

\usepackage{amsmath,amssymb,amsfonts,amsthm,latexsym,graphicx,multirow,enumerate,hyperref,todonotes}

 \newcommand\Aut{\mathrm{Aut}}
 
 \newcommand\calB{\mathcal{B}}

\newcommand\calD{\mathcal{D}}

\newcommand\Nor{\mathbb{N}}

        \newcommand\Sym{\mathrm{Sym}} 

\newcommand\val{\mathrm{val}}

\newcommand\ZZ{\mathbb{Z}}

\newtheorem{theorem}{Theorem}[section]
\newtheorem{lemma}[theorem]{Lemma}

\newtheorem{problem}[theorem]{Problem}

\theoremstyle{definition}
\newtheorem{definition}[theorem]{Definition}

\newtheorem{example}[theorem]{Example}
\newtheorem{remark}[theorem]{Remark}

\long\def\delete#1{}

\usepackage{color}

\definecolor{Blue}{rgb}{0,0,1}
\definecolor{Red}{rgb}{1,0,0}
\definecolor{DarkGreen}{rgb}{0,0.6,0}
\definecolor{DarkYellow}{rgb}{1,1,0.2}
\definecolor{DarkPurple}{rgb}{.6,0,1}

\newcommand{\red}[1]{{\color{Red}{#1}}}

\newcommand{\zhou}[1]{\red{\bf [by Sanming: #1]}}

\usepackage{xcolor}
\usepackage[normalem]{ulem}

\begin{document}
\openup 0.5\jot

\title{Stability of graph pairs}

\author[Qin]{Yan-Li Qin}
\address{School of Statistics\\Capital University of Economics and Business\\Beijing, 100070\\ P. R. China}
\email{ylqin@cueb.edu.cn}

\author[Xia]{Binzhou Xia}
\address{School of Mathematics and Statistics\\The University of Melbourne\\Parkville, VIC 3010\\Australia}
\email{binzhoux@unimelb.edu.au}

\author[Zhou]{Jin-Xin Zhou}
\address{Department of Mathematics\\Beijing Jiaotong University\\Beijing, 100044\\ P. R. China}
\email{jxzhou@bjtu.edu.cn}

\author[Zhou]{Sanming Zhou}
\address{School of Mathematics and Statistics\\The University of Melbourne\\Parkville, VIC 3010\\Australia}
\email{sanming@unimelb.edu.au}

\maketitle

\begin{abstract}
We start up the study of the stability of general graph pairs. This notion is a generalization of the concept of the stability of graphs. We say that a pair of graphs $(\Gamma,\Sigma)$ is stable if $\Aut(\Gamma\times\Sigma) \cong \Aut(\Gamma)\times\Aut(\Sigma)$ and unstable otherwise, where $\Gamma\times\Sigma$ is the direct product of $\Gamma$ and $\Sigma$. An unstable graph pair $(\Gamma,\Sigma)$ is said to be a nontrivially unstable graph pair if $\Gamma$ and $\Sigma$ are connected coprime graphs, at least one of them is non-bipartite, and each of them has the property that different vertices have distinct neighbourhoods. We obtain necessary conditions for a pair of graphs to be stable. We also give a characterization of a pair of graphs $(\Gamma, \Sigma)$ to be nontrivially unstable in the case when both graphs are connected and regular with coprime valencies and $\Sigma$ is vertex-transitive. This characterization is given in terms of the $\Sigma$-automorphisms of $\Gamma$, which are a new concept introduced in this paper as a generalization of both automorphisms and two-fold automorphisms of a graph.

\textit{Key words:} stable graph; stable graph pair; direct product of graphs
\end{abstract}

\section{Introduction}

All graphs considered in this paper are finite, undirected and simple, unless stated otherwise. As usual, for a graph $\Gamma$ we use $V(\Gamma)$, $E(\Gamma)$ and $\Aut(\Gamma)$ to denote its vertex set, edge set and full automorphism group, respectively, and we use $\val(\Gamma)$ to denote the valency of $\Gamma$ if $\Gamma$ is regular. For a vertex $u$ of $\Gamma$, the \emph{neighborhood} of $u$ in $\Gamma$, denoted by $N_\Gamma(u)$, is the set of vertices adjacent to $u$ in $\Gamma$. For two adjacent vertices $u, v$ in a graph, the edge between them is denoted by the unordered pair $\{u, v\}$. 

Let $\Gamma$ and $\Sigma$ be graphs. The \emph{direct product} of $\Gamma$ and $\Sigma$, denoted by $\Gamma\times\Sigma$, is the graph with vertex set $V(\Gamma)\times V(\Sigma)$ such that two vertices $(u,x),(v,y)\in V(\Gamma)\times V(\Sigma)$ are adjacent if and only if $u$ and $v$ are adjacent in $\Gamma$ and $x$ and $y$ are adjacent in $\Sigma$. Clearly,
\begin{equation}\label{eq1}
\Aut(\Gamma)\times\Aut(\Sigma) \lesssim \Aut(\Gamma\times\Sigma).
\end{equation}
Herein and in the sequel we use $X \lesssim Y$ to indicate that $X$ is isomorphic to a subgroup of $Y$, and $\times$ on the left-hand side denotes the direct product of groups.
In the literature much attention has been paid to the automorphism group $\Aut(\Gamma\times\Sigma)$ of $\Gamma\times\Sigma$. In particular, the question of when the equality in~\eqref{eq1} holds has attracted considerable interest (see, for example, \cite{Dorfler1974, HIK2011}). In line with this we introduce the following definition.

\begin{definition}
A graph pair $(\Gamma,\Sigma)$ is called \emph{stable} if $\Aut(\Gamma\times\Sigma) \cong \Aut(\Gamma)\times\Aut(\Sigma)$ and \emph{unstable} otherwise. 
\end{definition}

Note that $(\Gamma,\Sigma)$ is stable if and only if $(\Sigma,\Gamma)$ is stable. The notion of the stability of graph pairs generalizes the concept of the stability of graphs \cite{MSZ1989}, in the sense that a graph $\Gamma$ is stable if and only if the graph pair $(\Gamma,K_2)$ is stable, where $K_2$ is the complete graph with two vertices. Introduced by Maru\v{s}i\v{c} et al. \cite{MSZ1989} in the language of symmetric $(0, 1)$ matrices, the stability of graphs has been studied extensively (see, for example, \cite{QXZ2019,Surowski2001,Surowski2003,Wilson2008}) owing to its close connections with regular embeddings of canonical double covers \cite{NS1996}, two-fold automorphisms of graphs \cite{LMS2015}, and generalized Cayley graphs \cite{MSZ1992}. For example, the stability of circulant graphs was studied by Wilson in \cite{Wilson2008}, and an open question in \cite{Wilson2008} about the stability of arc-transitive circulant graphs was answered and an infinite family of counterexamples to a conjecture of Maru\v{s}i\v{c} et al. \cite{MSZ1989} was constructed by Qin et al. in \cite{QXZ2019}. A conjecture of Wilson \cite{Wilson2008} about the stability of generalized Petersen graphs was recently proved by Qin et al. in \cite{QXZ2020}. 

A graph $\Gamma$ is said to be \emph{$R$-thick} \cite{HIK2011} if there exist distinct vertices $u, v$ of $\Gamma$ such that $N_\Gamma(u)=N_\Gamma(v)$. Graphs that are not $R$-thick are said to be \emph{$R$-thin} \cite{HIK2011} or \emph{vertex-determining \cite{QXZ2019, Wilson2008}}. A graph is said to be \emph{prime} (with respect to the direct product) if it has order greater than $1$ and cannot be represented as a direct product of two graphs of smaller orders, where the \emph{order} of a graph is defined as its number of vertices. Since the direct product of graphs is an associative and commutative operation, the direct product of more than two graphs is well defined up to isomorphism. An expression $\Gamma\cong\Gamma_1\times\Gamma_2\times\cdots\times\Gamma_k$ with each $\Gamma_i$ prime is called a \emph{prime factorization} of $\Gamma$ (with respect to the direct product). It is well known \cite[Theorem 8.17]{HIK2011} that up to permutation of factors any non-bipartite graph with order greater than $1$ has a unique prime factorization. Two graphs are called \emph{coprime} (with respect to the direct product) if they do not have any common factor of order greater than $1$. In particular, any two graphs of coprime orders must be coprime.

\begin{remark}\label{coprime-valencies}
If two graphs have a common factor with respect to the direct product, then their valencies must have a common divisor greater than 1. Thus regular graphs with coprime valencies must be coprime. 
\end{remark}

Our first main result in this paper gives necessary conditions for a graph pair to be stable.

\begin{theorem}\label{thm1}
Let $(\Gamma,\Sigma)$ be a stable pair of graphs. Then $\Gamma$ and $\Sigma$ are coprime $R$-thin graphs. Moreover, if in addition both $\Aut(\Gamma)$ and $\Aut(\Sigma)$ are nontrivial groups, then both $\Gamma$ and $\Sigma$ are connected and at least one of them is non-bipartite.
\end{theorem}

In this paper we are only interested in graphs with nontrivial automorphism groups. Under this assumption Theorem~\ref{thm1} implies that in studying the stability of graph pairs we can focus on those pairs $(\Gamma,\Sigma)$ such that $\Gamma$ and $\Sigma$ are connected coprime $R$-thin graphs and at least one of them is non-bipartite. This fact motivates the following definition.

\begin{definition}
An unstable graph pair $(\Gamma,\Sigma)$ is said to be \emph{nontrivially unstable} if $\Gamma$ and $\Sigma$ are connected coprime $R$-thin graphs and at least one of them is non-bipartite.
\end{definition}

Since a graph not coprime to $K_2$ is necessarily bipartite, we see that $(\Gamma, K_2)$ is nontrivially unstable if and only if $\Gamma$ is a non-bipartite connected $R$-thin unstable graph. Such a graph $\Gamma$ is called \emph{nontrivially unstable} by Wilson in~\cite{Wilson2008} in a study of the stability of graphs. So the above definition of nontrivially unstable graph pairs generalizes the concept of nontrivially unstable graphs.

\begin{remark}
It can be easily shown that $\Gamma\times\Sigma$ is $R$-thin if and only if both $\Gamma$ and $\Sigma$ are $R$-thin (\cite[Lemma~2.3]{QXZ2019}, stated as Lemma \ref{cvn}(b) in the present paper), and $\Gamma\times\Sigma$ is connected if and only if both $\Gamma$ and $\Sigma$ are connected and at least one of them is non-bipartite (\cite[Theorem 5.9]{HIK2011}, stated as Lemma \ref{cvn}(a) in the present paper). Therefore, our definition of graph pairs $(\Gamma,\Sigma)$ being nontrivially unstable is equivalent to requiring that $\Gamma$ and $\Sigma$ are coprime graphs with $\Gamma\times\Sigma$ connected and $R$-thin.
\end{remark}

Needless to say, the following problem is of central importance to the study of the stability of graph pairs.

\begin{problem}\label{prob}
Characterize nontrivially unstable pairs of graphs $(\Gamma,\Sigma)$ with both $\Aut(\Sigma)$ and $\Aut(\Gamma)$ nontrivial.
\end{problem}



We will study this problem in the case when $\Gamma$ and $\Sigma$ are regular graphs of coprime valencies. Our study is motivated by orientably regular embeddings of the \emph{canonical double cover} $\Gamma\times K_2$ of a given graph $\Gamma$. More precisely, it was shown by Nedela and \v Skoviera~\cite{NS1996} that for any stable graph $\Gamma$ (that is, for any stable graph pair $(\Gamma,K_2)$), all orientably regular embeddings of $\Gamma\times K_2$ can be described in terms of orientably regular embeddings of $\Gamma$. As a natural extension, one would expect that for a stable pair $(\Gamma,\Sigma)$ of graphs we may be able to describe all orientably regular embeddings of $\Gamma\times\Sigma$ in terms of orientably regular embeddings of $\Gamma$ and $\Sigma$. In this regard it has been proved by Chen \cite{Chen} that, if $(\Gamma,\Sigma)$ is a stable pair of regular graphs such that $\Gamma\times\Sigma$ has an orientably regular embedding, then the valencies of $\Gamma$ and $\Sigma$ must be coprime. It is thus natural to impose the extra condition $\gcd(\val(\Gamma),\val(\Sigma))=1$ when studying Problem~\ref{prob}, and we will do so in this paper. Note that this condition is satisfied by $(\Gamma, K_2)$ for any regular graph $\Gamma$.

Let $\Gamma$ be a graph. A pair of permutations $(\alpha,\beta)$ of $V(\Gamma)$ is called a \emph{two-fold automorphism} of $\Gamma$ if for all $u,v\in V(\Gamma)$, $\{u,v\} \in E(\Gamma)$ if and only if $\{u^\alpha, v^\beta\} \in E(\Gamma)$. A two-fold automorphism $(\alpha,\beta)$ is said to be \emph{nontrivial} if in addition $\alpha\neq\beta$. It is proved in~\cite[Theorem~3.2]{LMS2015} that a graph is unstable if and only if it has a nontrivial two-fold automorphism. The second main result in our paper, Theorem~\ref{main-theorem2} below, generalizes this result to the setting of nontrvially unstable graph pairs when $\Gamma$ is regular. To present our result we need the following definition.

\begin{definition}
\label{def:sigma-auto}
Let $\Gamma$ and $\Sigma$ be graphs with $V(\Sigma)=\{1,\dots,n\}$, and let $\alpha_1,\dots,\alpha_n$ be permutations of $V(\Gamma)$. We say that the $n$-tuple $(\alpha_1,\dots,\alpha_n)$ is a \emph{$\Sigma$-automorphism} of $\Gamma$ if for all $u,v\in V(\Gamma)$, $\{u, v\} \in E(\Gamma)$ if and only if $\{u^{\alpha_i}, v^{\alpha_j}\} \in E(\Gamma)$ for all $i,j\in V(\Sigma)$ with $\{i, j\} \in E(\Sigma)$. Such a $\Sigma$-automorphism $(\alpha_1,\dots,\alpha_n)$ of $\Gamma$ is said to be \emph{nondiagonal} if there exists at least one pair of vertices $i,j\in V(\Sigma)$ such that $\alpha_i\neq\alpha_j$. 
\end{definition}

It is readily seen that a two-fold automorphism of a graph $\Gamma$ is exactly a $K_2$-automorphism of $\Gamma$, and a nontrivial two-fold automorphism of $\Gamma$ is precisely a nondiagonal $K_2$-automorphism of $\Gamma$. As a side note, we mention that the definition above applies when $\Gamma$ and $\Sigma$ are pseudographs. (A pseudograph is a graph in which both loops and multiple edges are permitted.) And we observe that, if $\Sigma$ is the pseudograph $K_{1}^{\circ}$ with only one vertex and one self-loop, then a $K_{1}^{\circ}$-automorphism of $\Gamma$ is an automorphism of $\Gamma$ in the usual sense. Nevertheless, in this paper we only consider the case when both $\Gamma$ and $\Sigma$ are simple graphs with order at least two.

\delete
{
As will be seen in the next section, all $\Sigma$-automorphisms of $\Gamma$ form a group under the coordinate-wise composition of permutations. 
}

The second main result in this paper, presented below, settles Problem \ref{prob} in the case when $\Gamma$ and $\Sigma$ are regular with coprime valencies and $\Sigma$ is vertex-transitive.

\begin{theorem}
\label{main-theorem2}
Let $\Gamma$ be a connected regular graph and $\Sigma$ a connected vertex-transitive graph such that $\val(\Gamma)$ and $\val(\Sigma)$ are coprime. Suppose that both $\Gamma$ and $\Sigma$ are $R$-thin and at least one of them is non-bipartite. Then $(\Gamma,\Sigma)$ is nontrvially unstable if and only if at least one $\Sigma$-automorphism of $\Gamma$ is nondiagonal.
\end{theorem}

The rest of this paper is organized as follows. In the next section we will introduce the $\Sigma$-automorphism group of $\Gamma$ and two subgroups (see Definitions \ref{def:Q} and \ref{def:P} in the next section) of $\Aut(\Gamma\times\Sigma)$, and study connections between these groups and the stability of $(\Gamma,\Sigma)$. The proof of Theorem~\ref{thm1} will be given in Section~\ref{unstable}. In Section~\ref{GraphsCoprime}, we will prove a number of lemmas concerning pairs of graphs with coprime valencies. These lemmas will be used to prove Theorem~\ref{main-theorem2} in Section~\ref{ProofTh2}. We will conclude the paper with some remarks and questions in Section \ref{Open}.

\section{\texorpdfstring{$\Sigma$-automorphism group of $\Gamma$}{Sigma-automorphism group of Gamma}}\label{sec3}

Throughout this section $\Gamma$ and $\Sigma$ are graphs with $V(\Sigma)=\{1,\dots,n\}$ and $n>1$. 

\subsection{$\Aut_\Sigma(\Gamma)$, $P(\Gamma,\Sigma)$ and $Q(\Gamma,\Sigma)$}

It is not difficult to see that the set of all $\Sigma$-automorphisms of $\Gamma$ with multiplication defined by
\[
(\alpha_1, \dots, \alpha_n)(\beta_1, \dots, \beta_n)=(\alpha_1\beta_1, \dots, \alpha_n\beta_n)
\]
is a group. We call this group the \emph{$\Sigma$-automorphism group} of $\Gamma$ and denote it by $\Aut_\Sigma(\Gamma)$. Note that $(\alpha,\alpha,\dots, \alpha)\in\Aut_\Sigma(\Gamma)$ if and only if $\alpha\in\Aut(\Gamma)$.
Hence 
\begin{equation}
\label{subgp}
\Aut(\Gamma)\lesssim\Aut_\Sigma(\Gamma).
\end{equation}
Moreover,  
\begin{equation}\label{allequal}
\Aut(\Gamma)\cong\Aut_\Sigma(\Gamma) \Leftrightarrow \alpha_1=\cdots=\alpha_n\text{ for each }(\alpha_1,\dots,\alpha_n)\in\Aut_\Sigma(\Gamma).
\end{equation}

\begin{definition}
\label{def:Q}
Define
$$
Q(\Gamma,\Sigma) = \{\sigma \in \Aut(\Gamma\times\Sigma): (V(\Gamma)\times\{i\})^{\sigma} = V(\Gamma)\times\{i\} \text{ for each } i\in V(\Sigma)\}.
$$
\end{definition}

Note that $Q(\Gamma,\Sigma)$ is a subgroup of $\Aut(\Gamma\times\Sigma)$. Denote the projections from $V(\Gamma \times\Sigma)$ to $V(\Gamma)$ and $V(\Sigma)$ by $\pi_\Gamma$ and $\pi_\Sigma$, respectively. In other words, 
$$
(u, i)^{\pi_\Gamma} = u \text{ and } (u, i)^{\pi_\Sigma} = i, \text{ for } (u,i) \in V(\Gamma \times\Sigma).
$$  
  
\begin{lemma}\label{NF}
Let $\Gamma$ and $\Sigma$ be graphs. Then 
$$
Q(\Gamma,\Sigma) \cong \Aut_\Sigma(\Gamma).
$$
\end{lemma}

\begin{proof}
Set $V(\Sigma)=\{1,\dots,n\}$. Define
\[
f\colon Q(\Gamma,\Sigma)\rightarrow\Sym(V(\Gamma))\times\cdots\times\Sym(V(\Gamma)), \quad \sigma\mapsto (\alpha_1, \dots, \alpha_n)
\]
such that $(u,i)^{\sigma\pi_\Gamma}=(u,i)^{\pi_\Gamma\alpha_i}$ for each $(u,i)\in V(\Gamma\times\Sigma)$.
For any $\sigma, \tau \in Q(\Gamma,\Sigma)$, let $f(\sigma)=(\alpha_1, \dots, \alpha_n)$ and
$f(\tau)=(\beta_1, \dots, \beta_n)$. It is straightforward to verify that $(u,i)^{\sigma \tau \pi_\Gamma}=(u,i)^{\pi_\Gamma(\alpha_i\beta_i)}$
for each $(u,i)\in V(\Gamma\times\Sigma)$. Hence
\begin{align*}
f(\sigma)f(\tau)&=(\alpha_1, \dots, \alpha_n)(\beta_1, \dots, \beta_n)\\
&=(\alpha_1\beta_1, \dots, \alpha_n\beta_n)\\
&=f(\sigma \tau).
\end{align*}
Thus $f$ is a group homomorphism. Since $f$ is injective, we have $Q(\Gamma,\Sigma)\cong f(Q(\Gamma,\Sigma))$.

Let $(\alpha_1, \dots, \alpha_n)\in\Aut_\Sigma(\Gamma)$. Define $\sigma$ by $(u, i)^{\sigma}=(u^{\alpha_i}, i)$ for all $i\in V(\Sigma)$. Then $\sigma\in Q(\Gamma,\Sigma)$ and $f(\sigma)=(\alpha_1, \dots, \alpha_n)$. Thus $\Aut_\Sigma(\Gamma)\subseteq f(Q(\Gamma,\Sigma))$.

Conversely, let $\sigma\in Q(\Gamma,\Sigma)$ and $f(\sigma)=(\alpha_1, \dots, \alpha_n)$. Then for any $i,j\in V(\Sigma)$ with $\{i, j\} \in E(\Sigma)$, we have
\begin{align*}
\{u, v\}\in E(\Gamma) & \Leftrightarrow \{(u,i), (v, j)\}\in E(\Gamma\times\Sigma)\\
& \Leftrightarrow \{(u,i)^{\sigma}, (v, j)^{\sigma}\}=\{(u^{\alpha_i}, i), (v^{\alpha_j}, j)\}\in E(\Gamma\times\Sigma)\\
& \Leftrightarrow \{u^{\alpha_i}, v^{\alpha_j}\}\in E(\Gamma),
\end{align*}
whence $\{u, v\} \in E(\Gamma)$ if and only if $\{u^{\alpha_i}, v^{\alpha_j}\} \in E(\Gamma)$. Thus $f(\sigma)=(\alpha_1, \dots, \alpha_n)\in\Aut_\Sigma(\Gamma)$, and so $f(Q(\Gamma,\Sigma))\subseteq\Aut_\Sigma(\Gamma)$. Therefore, $Q(\Gamma,\Sigma)\cong f(Q(\Gamma,\Sigma))=\Aut_\Sigma(\Gamma)$, completing the proof.
\end{proof}

Obviously, $\{V(\Gamma)\times\{i\}\mid i\in V(\Sigma)\}$ is a partition of $V(\Gamma\times\Sigma)$. 

\begin{definition}
\label{def:P}
Define $P(\Gamma,\Sigma)$ to be the set of elements of $\Aut(\Gamma\times\Sigma)$ that leave the partition $\{V(\Gamma)\times\{i\}\mid i\in V(\Sigma)\}$ invariant. 
\end{definition}

That is, $P(\Gamma,\Sigma)$ is the setwise stabilizer of $\{V(\Gamma)\times\{i\}\mid i\in V(\Sigma)\}$ under $\Aut(\Gamma\times\Sigma)$. Hence it is a subgroup of $\Aut(\Gamma\times\Sigma)$. Of course $P(\Gamma,\Sigma)$ induces an action on $\{V(\Gamma)\times\{i\}\mid i\in V(\Sigma)\}$, and the kernel of this action is exactly $Q(\Gamma,\Sigma)$. Hence 
$$
P(\Gamma,\Sigma)/Q(\Gamma,\Sigma)\cong\Aut(\Sigma).
$$ 
This together with Lemma~\ref{NF} implies the following result. 
 
\begin{lemma}\label{stable-NF}
Let $\Gamma$ and $\Sigma$ be graphs. Then 
$$
P(\Gamma,\Sigma)\cong\Aut_\Sigma(\Gamma)\rtimes\Aut(\Sigma).
$$
\end{lemma}

\subsection{$P(\Gamma,\Sigma)$ and the stability of $(\Gamma,\Sigma)$}

\begin{lemma}
\label{notP}
Let $\Gamma$ and $\Sigma$ be graphs. If $(\Gamma,\Sigma)$ is stable, then 
$$
\Aut(\Gamma\times\Sigma)=P(\Gamma,\Sigma).
$$
\end{lemma}

\begin{proof}
By \eqref{subgp} and Lemma~\ref{stable-NF}, we have
$$
\Aut(\Gamma)\times\Aut(\Sigma)\lesssim\Aut_\Sigma(\Gamma)\rtimes\Aut(\Sigma)\cong P(\Gamma,\Sigma)\leqslant\Aut(\Gamma\times\Sigma).
$$
Thus, if $(\Gamma,\Sigma)$ is stable, then $\Aut(\Gamma\times\Sigma)=P(\Gamma,\Sigma)$.
\end{proof}
 
\begin{lemma}
\label{nontrivial}
Let $\Gamma$ and $\Sigma$ be graphs, where $V(\Sigma)=\{1,\dots,n\}$. 
\begin{enumerate}[\rm (a)]
\item If at least one $\Sigma$-automorphism of $\Gamma$ is nondiagonal, then $(\Gamma,\Sigma)$ is unstable.
\item If $\Aut(\Gamma\times\Sigma)=P(\Gamma,\Sigma)$, then $(\Gamma,\Sigma)$ is unstable if and only if at least one $\Sigma$-automorphism of $\Gamma$ is nondiagonal.
\end{enumerate}
\end{lemma}

\begin{proof}
(a) Suppose that there is at least one nondiagonal $\Sigma$-automorphism of $\Gamma$. Then we see from \eqref{subgp} and \eqref{allequal} that $\Aut(\Gamma)$ is isomorphic to a proper subgroup of $\Aut_\Sigma(\Gamma)$. Combining this with Lemma~\ref{stable-NF}, we obtain
\[
|\Aut(\Gamma)\times\Aut(\Sigma)|<|\Aut_\Sigma(\Gamma)\rtimes\Aut(\Sigma)|=|P(\Gamma,\Sigma)|\leqslant|\Aut(\Gamma\times\Sigma)|.
\]
Thus $\Aut(\Gamma\times\Sigma)\ncong\Aut(\Gamma)\times\Aut(\Sigma)$ and so $(\Gamma,\Sigma)$ is unstable.

(b) Suppose that $\Aut(\Gamma\times\Sigma)=P(\Gamma,\Sigma)$. Then $\Aut(\Gamma\times\Sigma)\cong\Aut_\Sigma(\Gamma)\rtimes\Aut(\Sigma)$ by Lemma \ref{stable-NF}. Thus $(\Gamma,\Sigma)$ is unstable if and only if $\Aut_\Sigma(\Gamma)\ncong \Aut(\Gamma)$, which, by \eqref{subgp} and \eqref{allequal}, is true if and only if at least one $\Sigma$-automorphism of $\Gamma$ is nondiagonal.
\end{proof}

\section{Proof of Theorem~\ref{thm1}}\label{unstable}

\subsection{Preparation}

We need a few lemmas before we can prove Theorem~\ref{thm1}. First, the following known results will be used in our proofs of Theorems \ref{thm1} and \ref{main-theorem2}. 

\begin{lemma}
\label{cvn}
Let $\Gamma$ and $\Sigma$ be graphs. 
\begin{enumerate}[{\rm (a)}]
\item If $\Gamma$ and $\Sigma$ are connected with order at least $2$, then $\Gamma \times \Sigma$ is connected if at least one of $\Gamma$ and $\Sigma$ is non-bipartite, and $\Gamma \times \Sigma$ has exactly two components if both $\Gamma$ and $\Sigma$ are bipartite (\cite{Weichsel62}; see also \cite[Theorem~5.9]{HIK2011}).
\item $\Gamma \times \Sigma$ is $R$-thin if and only if both $\Gamma$ and $\Sigma$ are $R$-thin (\cite[Lemma~2.3]{QXZ2019}).
\item $\Gamma \times \Sigma$ is non-bipartite if and only if both $\Gamma$ and $\Sigma$ are non-bipartite (\cite[Exercise 8.13]{HIK2011}). 
\end{enumerate}
\end{lemma}
 
\begin{lemma}\label{VD}
Let $\Gamma$ and $\Sigma$ be graphs. If $(\Gamma,\Sigma)$ is stable, then both $\Gamma$ and $\Sigma$ are $R$-thin.
\end{lemma}

\begin{proof}
We prove the contrapositive of the statement. Suppose that at least one of $\Gamma$ and $\Sigma$ is $R$-thick. Without loss of generality we may assume that $\Gamma$ is $R$-thick. Then there exist two distinct vertices $u$ and $v$ of $\Gamma$ such that $N_{\Gamma}(u) = N_{\Gamma}(v)$. Let $\alpha$ be the permutation of $V(\Gamma)$ which swaps $u$ and $v$ and fixes each vertex in $V(\Gamma) \setminus \{u, v\}$, and let $\tau=(\alpha, 1, 1, \dots, 1)$. Then $\tau$ is a $\Sigma$-automorphism of $\Gamma$. Moreover, $\tau$ is nondiagonal as $\alpha\neq1$. Thus, by Lemma \ref{nontrivial}(a), $(\Gamma,\Sigma)$ is unstable.  
\end{proof}

\delete
{
Let $\Gamma$ and $\Sigma$ be graphs. It can be proved that, if $\Gamma$ contains a connected component $\Gamma_1$ such that $(\Gamma_1, \Sigma)$ is unstable, then $(\Gamma, \Sigma)$ is unstable. In fact, for any $\alpha \in \Aut(\Gamma_1\times\Sigma) \setminus (\Aut(\Gamma_1)\times\Aut(\Sigma))$, the permutation of $V(\Gamma\times\Sigma)$ which agrees with $\alpha$ on $V(\Gamma_1\times\Sigma)$ and fixes each vertex in $V(\Gamma\times\Sigma) \setminus V(\Gamma_1\times\Sigma )$ is in $\Aut(\Gamma\times \Sigma) \setminus (\Aut(\Gamma)\times\Aut(\Sigma))$.

\begin{lemma}\label{disconnected2}
Let $\Sigma$ be a connected graph and $\Gamma$ a disconnected graph. If $\Gamma$ contains a connected component $\Gamma_1$ such that $(\Gamma_1, \Sigma)$ is unstable, then $(\Gamma, \Sigma)$ is unstable. \zhou{It seems that this lemma is not used. Also, from the proof it seems that the result is also true if $\Sigma$ is disconnected.}
\end{lemma}

\begin{proof}
Since $(\Gamma_1, \Sigma)$ is unstable, we have $\Aut(\Gamma_1)\times\Aut(\Sigma)<\Aut(\Gamma_1\times\Sigma)$. So there exists $\alpha\in \Aut(\Gamma_1\times\Sigma)$ such that $\alpha\notin\Aut(\Gamma_1)\times\Aut(\Sigma)$. Let $\beta$ be the permutation of $V(\Gamma\times\Sigma)$ which fixes each vertex in $V(\Gamma\times\Sigma)\setminus V(\Gamma_1\times\Sigma )$ and agrees with $\alpha$ on $V(\Gamma_1\times\Sigma)$. It is straightforward to verify that $\beta\in\Aut(\Gamma\times \Sigma)$. On the other hand, we have $\beta\notin\Aut(\Gamma)\times\Aut(\Sigma)$ as $\alpha\notin\Aut(\Gamma_1)\times\Aut(\Sigma)$. Thus $(\Gamma, \Sigma)$ is unstable.
\end{proof}
}

\begin{lemma}
\label{disconnected}
Let $\Gamma$ and $\Sigma$ be graphs. If one of them is disconnected and the other has a nontrivial automorphism group, then $(\Gamma,\Sigma)$ is unstable.
\end{lemma}

\begin{proof}
Without loss of generality we may assume that $\Gamma$ is disconnected and $\Aut(\Sigma)\neq1$. Take a connected component $\Gamma_1$ of $\Gamma$ and an element $\alpha \ne 1$ of $\Aut(\Sigma)$. Let $\sigma$ be the permutation of $V(\Gamma\times\Sigma)$ which fixes each vertex in $V(\Gamma\times\Sigma)\setminus V(\Gamma_1\times\Sigma)$ and permutes the vertices in $V(\Gamma_1\times\Sigma)$ in the following way:
\[
(u,i)^\sigma=(u,i^\alpha) \quad \text{for each $(u,i)\in V(\Gamma_1\times\Sigma)$}.
\]
It is straightforward to verify that $\sigma\in\Aut(\Gamma\times \Sigma)\setminus P(\Gamma,\Sigma)$. So $(\Gamma,\Sigma)$ is unstable by Lemma \ref{notP}.
\end{proof}

In regard to Lemma \ref{disconnected}, there exist both stable pairs $(\Gamma,\Sigma)$ and unstable pairs $(\Gamma,\Sigma)$ with $\Gamma$ disconnected and $\Aut(\Sigma)=1$, as illustrated by the following example.

\begin{example} 
Let $\Gamma_1$ be the graph with $V(\Gamma_1)=\{1,2,3,4,5\}$ and $$E(\Gamma_1)=\left\{\{1,2\},\{3,4\},\{4,5\},\{3,5\}\right\}.$$ Let $\Gamma_2$ be the graph with $V(\Gamma_2)=\{1,2,3,4,5,6,7,8\}$ and $$E(\Gamma_2)=\{\{1,2\},\{2,3\},\{3,4\},\{4,5\},\{3,6\},\{4,6\},\{7,8\}\}.$$ Let $\Sigma$ be the graph with $V(\Sigma)=\{1,2,3,4,5,6\}$ and $$E(\Sigma)=\{\{1,2\},\{2,3\},\{3,4\},\{4,5\},\{3,6\},\{4,6\}\}.$$ Then both $\Gamma_1$ and $\Gamma_2$ are disconnected and $\Aut(\Sigma)=1$. Computation in \textsc{Magma}~\cite{BCP1997} shows that $|\Aut(\Gamma_1 \times \Sigma)|=|\Aut(\Gamma_1)|=12$, $|\Aut(\Gamma_2)|=2$, and $|\Aut(\Gamma_2 \times \Sigma)|=4$. Hence $(\Gamma_1,\Sigma)$ is stable but $(\Gamma_2,\Sigma)$ is unstable.
\end{example}
 
As usual, for a graph $\Gamma$ and a subset $U \subseteq V(\Gamma)$, we use $\langle U \rangle$ to denote the subgraph of $\Gamma$ induced by $U$, namely the graph with vertex set $U$ in which two vertices $u, v \in U$ are adjacent if and only if they are adjacent in $\Gamma$.
 
\begin{lemma}
\label{bipartite}
Let $\Gamma$ and $\Sigma$ be connected bipartite graphs. If both $\Aut(\Gamma)$ and $\Aut(\Sigma)$ are nontrivial groups, then $(\Gamma,\Sigma)$ is unstable.
\end{lemma}

\begin{proof}
Let $\{B_1,B_2\}$ be the bipartition of $\Gamma$ and $\{C_1,C_2\}$ the bipartition of $\Sigma$. Set
\[
U_1 = (B_1 \times C_1) \cup (B_2 \times C_2) \text{ and }
U_2 = (B_1 \times C_2) \cup (B_2 \times C_1).
\]
Then $\langle U_1\rangle$ and $\langle U_2\rangle$ are two connected components of $\Gamma\times\Sigma=\langle U_1 \cup U_2\rangle$. Since both $\Aut(\Gamma)$ and $\Aut(\Sigma)$ are nontrivial, we have
\begin{equation}
\label{aa}
|\Aut(\Gamma\times\Sigma)|\geqslant|\Aut(\Gamma)|\cdot|\Aut(\Sigma)|\geqslant2\cdot 2=4. 
\end{equation}
If both $\Aut(\langle U_1\rangle)$ and $\Aut(\langle U_2\rangle)$ are trivial, then
$$
\Aut(\Gamma\times\Sigma)=\Aut(\langle U_1 \cup U_2\rangle)=
\begin{cases}
1, \quad&\text{if $\langle U_1\rangle\ncong\langle U_2\rangle$}\\
\ZZ_2, \quad&\text{if $\langle U_1\rangle\cong\langle U_2\rangle$},
\end{cases}
$$
which contradicts \eqref{aa}. Hence at least one of $\Aut(\langle U_1\rangle)$ and $\Aut(\langle U_2\rangle)$ is nontrivial. Without loss generality we may assume that $\Aut(\langle U_2\rangle)$ is nontrivial. Take an element $1\neq\alpha\in\Aut(\langle U_2\rangle)$. Let $\sigma$ be the permutation of $V(\Gamma\times\Sigma)$ defined by
\begin{equation}\label{sa}
(u,i)^\sigma=
\begin{cases}
(u,i), \quad & \text{if $(u,i)\in U_1$}\\
(u,i)^{\alpha}, \quad & \text{if $(u,i)\in U_2$}
\end{cases}
\end{equation}
for $(u,i)\in V(\Gamma\times\Sigma)$. Then $\sigma\in\Aut(\Gamma\times\Sigma)$.

Suppose to the contrary that $(\Gamma,\Sigma)$ is stable. Then by Lemma \ref{notP}, $\Aut(\Gamma\times\Sigma)=P(\Gamma,\Sigma)$, and so $\sigma\in P(\Gamma,\Sigma)$.
Set $V(\Sigma)=\{1,\dots,n\}$ and let $\alpha_1,\dots,\alpha_n$ be permutations of $V(\Gamma)$ such that $u^{\alpha_i}=(u,i)^{\sigma\pi_\Gamma}$
for $u\in V(\Gamma)$ and $i\in V(\Sigma)$. Since $\sigma\in P(\Gamma,\Sigma)$, by \eqref{sa} we have for $(u,i)\in V(\Gamma\times\Sigma)$,
\begin{equation}\label{u1u2}
(u,i)^\sigma=
\begin{cases}
(u,i), \quad&\text{if $(u,i)\in U_1$}\\
(u^{\alpha_i},i), \quad&\text{if $(u,i)\in U_2$.}
\end{cases}
\end{equation}
Hence
\begin{align*}
\{u, v\} \in E(\Gamma) \Leftrightarrow &\, \{(u,i)^\sigma, (v,j)^\sigma\} \in E(\Gamma\times\Sigma) \text{ for any $i,j\in V(\Sigma)$ with $\{i, j\} \in E(\Sigma)$}\\
\Leftrightarrow &\, \{(u^{\alpha_i},i), (v^{\alpha_j},j)\} \in E(\Gamma\times\Sigma) \text{ for any $i,j\in V(\Sigma)$ with $\{i, j\} \in E(\Sigma)$}\\
\Leftrightarrow &\, \{u^{\alpha_i}, v^{\alpha_j}\} \in E(\Gamma) \text{ for any $i,j\in V(\Sigma)$ with $\{i, j\} \in E(\Sigma)$.}
\end{align*}
In other words, $(\alpha_1,\dots,\alpha_n)$ is a $\Sigma$-automorphism of $\Gamma$. Since $(\Gamma,\Sigma)$ is stable, by Lemma \ref{nontrivial}(a) we deduce that $\alpha_1=\cdots=\alpha_n$.
This together with~\eqref{u1u2} implies that $\sigma$ fixes each vertex of $\Gamma\times\Sigma$, whence $\sigma=1$. It then follows from~\eqref{sa} that $\alpha=1$. This contradiction shows that $(\Gamma,\Sigma)$ is unstable and the proof is complete.
\end{proof}

The following example shows that the condition that both $\Aut(\Gamma)$ and $\Aut(\Sigma)$ are nontrivial cannot be removed from Lemma \ref{bipartite} for otherwise the result may not be true.

\begin{example}\label{bipartite-example}
Let $\Gamma$ be the graph with $V(\Gamma)=\{1,2,3,4,5,6,7\}$ and $E(\Gamma)=\{\{1,2\},\{2,3\},\{3,4\},\{4,5\},\{2,6\},\{3,7\},\{6,7\}\}$. Let $\Sigma\cong K_2$ with $V(\Sigma)=\{a,b\}$. Then both $\Gamma$ and $\Sigma$ are bipartite and $\Aut(\Gamma)=1$. Let $\sigma$ be the permutation of $V(\Gamma\times\Sigma)$ which interchanges $(i,a)$ and $(i,b)$ for each $i\in V(\Gamma)$. Then $\Aut(\Gamma\times\Sigma)=\langle\sigma\rangle\cong\ZZ_2$. Hence $\Aut(\Gamma\times\Sigma)\cong\Aut(\Gamma)\times\Aut(\Sigma)$ and $(\Gamma,\Sigma)$ is stable.
\end{example}

Let $\Sigma$ and $\Gamma$ be graphs each with a nontrivial automorphism group. If one of them is $R$-thick or disconnected, then by Lemma \ref{VD} or \ref{disconnected} we know that $(\Gamma,\Sigma)$ is unstable. Part (b) of the following lemma determines the stability of $(\Gamma,\Sigma)$ when both $\Sigma$ and $\Gamma$ are connected, $R$-thin and non-bipartite.
 
\begin{lemma}
\label{coprime}
Let $\Gamma$ and $\Sigma$ be graphs. 
\begin{enumerate}[\rm (a)]
\item If $(\Gamma,\Sigma)$ is stable, then $\Gamma$ and $\Sigma$ are coprime.
\item If both $\Gamma$ and $\Sigma$ are connected, $R$-thin and non-bipartite, then $(\Gamma,\Sigma)$ is stable if and only if $\Gamma$ and $\Sigma$ are coprime.  
\end{enumerate}
\end{lemma}

\begin{proof}
(a) We prove the contrapositive of this statement. Suppose that $\Gamma$ and $\Sigma$ are not coprime. Then there exist graphs $\Gamma_1$, $\Sigma_1$ and $\Delta$ such that $\Gamma=\Gamma_1\times\Delta$, $\Sigma=\Sigma_1\times\Delta$ and $|V(\Delta)|>1$. So $\Gamma\times\Sigma=\Gamma_1\times\Delta\times\Sigma_1\times\Delta$. Let $\sigma$ be the permutation of $V(\Gamma\times\Sigma)=V(\Gamma_1\times\Delta\times\Sigma_1\times\Delta)$ defined by $(x,u,y,v)^\sigma=(x,v,y,u)$ for $(x,u,y,v)\in V(\Gamma_1\times\Delta\times\Sigma_1\times\Delta)$. Since $|V(\Delta)|>1$, it is straightforward to verify that $\sigma\in\Aut(\Gamma\times \Sigma)\setminus P(\Gamma,\Sigma)$. This together with Lemma \ref{notP} implies that $(\Gamma,\Sigma)$ is unstable. 

(b) The ``only if" part follows from (a), so it remains to prove the ``if" part. Since both $\Gamma$ and $\Sigma$ are connected, $R$-thin and non-bipartite, by Lemma~\ref{cvn}, so is $\Gamma\times\Sigma$. If $\Gamma$ and $\Sigma$ are coprime, then by~\cite[Theorem~8.18]{HIK2011}, we have $\Aut(\Gamma\times\Sigma)=\Aut(\Gamma)\times\Aut(\Sigma)$ and hence $(\Gamma,\Sigma)$ is stable.
\end{proof}

\subsection{Proof of Theorem~\ref{thm1}}

We are now ready to prove Theorem~\ref{thm1}.

\begin{proof}
Suppose that $(\Gamma,\Sigma)$ is stable. Then by Lemmas~\ref{VD} and~\ref{coprime}(a) we know that $\Gamma$ and $\Sigma$ are coprime $R$-thin graphs. Moreover, if both $\Aut(\Gamma)$ and $\Aut(\Sigma)$ are nontrivial, then by Lemmas~\ref{disconnected} and~\ref{bipartite} both $\Gamma$ and $\Sigma$ are connected and at least one of them is non-bipartite.
\end{proof}

\section{Pairs of regular graphs with coprime valencies}
\label{GraphsCoprime}

As a preparation for our proof of Theorem~\ref{main-theorem2}, we study pairs of regular graphs with coprime valencies in this section. As before, throughout this section $\Gamma$ and $\Sigma$ are graphs with $V(\Sigma) = \{1, \ldots, n\}$ and $n > 1$. As usual, the set of positive integers is denoted by $\Nor$. A key concept used in this section is the \emph{Boolean square} \cite{HIK2011} of a graph $\Delta$, denoted by $B(\Delta)$, which is the graph with vertex set $V(\Delta)$ and edge set $\{\{u,v\}\mid u, v \in V(\Delta), u \ne v, N_{\Delta}(u)\cap N_{\Delta}(v)\neq\emptyset\}$. 

\subsection{Notation}

For $(u,i), (v,j) \in V(\Gamma\times\Sigma)$, define
\[
f_{\Gamma,\Sigma}((u,i),(v,j)) = \frac{|N_{\Gamma\times\Sigma}((u,i))\cap N_{\Gamma\times\Sigma}((v,j))|}{|N_{\Gamma\times\Sigma}((u,i))|}.
\]
Define
$$
X_{\Gamma,\Sigma}(u,i) = \{(v,j)\in V(\Gamma\times\Sigma) \setminus \{(u,i)\} \mid \val(\Sigma)\cdot f_{\Gamma,\Sigma}((u,i),(v,j))\in\Nor\}
$$
and
$$
Y_{\Gamma,\Sigma}(u,i) = \{(v, j) \in X_{\Gamma,\Sigma}(u,i) \mid f_{\Gamma,\Sigma}((u,i),(v,j)) \geqslant f_{\Gamma,\Sigma}((u,i),(w,j))\ \forall w\in V(\Gamma)\}.
$$
We will abbreviate $f_{\Gamma,\Sigma}((u,i),(v,j))$, $X_{\Gamma,\Sigma}(u,i)$ and $Y_{\Gamma,\Sigma}(u,i)$ to $f((u,i),(v,j))$, $X(u,i)$ and $Y(u,i)$, respectively, when there is no danger of confusion.

The following lemma follows immediately from the definitions of $\Gamma \times \Sigma$ and $f((u,i),(v,j))$. 

\begin{lemma}\label{vs}
Let $\Gamma$ and $\Sigma$ be regular graphs, and let $(u,i),(v,j)\in V(\Gamma\times\Sigma)$. Then
\begin{equation}
\label{fs}
f((u,i),(v,j))=\frac{|N_\Gamma(u)\cap N_\Gamma(v)|}{\val(\Gamma)}\cdot\frac{|N_\Sigma(i)\cap N_\Sigma(j)|}{\val(\Sigma)}.
\end{equation}
\end{lemma}
 
Now assume that $\Sigma$ is vertex-transitive. Then the vertex-transitivity of $\Sigma$ implies that there exist integers $n_1 > n_2 > \cdots > n_t \ge 1$ such that for any $i \in V(\Sigma)$ and $x\in N_{B(\Sigma)}(i)$ we have 
$$
|N_\Sigma(x)\cap N_\Sigma(i)| \in \{n_1, n_2, \dots, n_t\} \text{ and } |N_{B(\Sigma)}(i)|>n_1.
$$ 
Set 
$$
D_0(i)=\{i\},
$$
$$
D_k(i)=\{x\in N_{B(\Sigma)}(i)\mid |N_\Sigma(x)\cap N_\Sigma(i)|=n_k, 1\leqslant n_k<|N_{B(\Sigma)}(i)|\}\ \text{ for } 1 \le k \le t,
$$
$$
D_{t+1}(i)=\{x\in V(\Sigma) \mid N_\Sigma(i)\cap N_\Sigma(x)=\emptyset\},
$$ 
and 
$$
D_s(i)=\emptyset\ \text{ for } s \geqslant t+2.
$$
Then
$$
V(\Sigma)=\cup_{k=1}^{t+1}D_k(i),\;
\val(\Sigma)=\sum_{k=1}^t|D_k(i)| \ \text{ and }\
N_{B(\Sigma)}(i)=\cup_{k=1}^tD_k(i).
$$
For each $(u,i)\in V(\Gamma\times\Sigma)$, define 
$$
X_0(u,i)=Y_0(u,i)=\{(u,i)\}
$$
and
$$
X_k(u,i)=X(u,i)\cap(V(\Gamma)\times D_k(i))\ \text{ for } k \ge 1.
$$
Define $Y_k(u,i)$ to be the set of elements $(v,j)$ of $X(u,i) \setminus (\cup_{m=0}^{k-1}X_m(u,i))$ such that  
$$
f((u,i),(v,j)) \geqslant f((u,i),(w,y))\ \text{ for all }\ (w,y) \in X(u,i) \setminus (\cup_{m=0}^{k-1}X_m(u,i)).
$$

\subsection{$X(u,i)$ and $X_k(u,i)$}

\begin{lemma}\label{X}
Let $\Gamma$ and $\Sigma$ be graphs. Then for any $(u, i) \in V(\Gamma \times \Sigma)$ and $\sigma \in \Aut(\Gamma\times\Sigma)$, we have 
$$
(X(u,i))^\sigma=X((u,i)^\sigma).
$$
\end{lemma}

\begin{proof}
For any $(u,i),(v,j)\in V(\Gamma\times\Sigma)$ and $\sigma\in\Aut(\Gamma\times\Sigma)$, we have
\[
(u,i)\neq(v,j)\Leftrightarrow(u,i)^\sigma\neq(v,j)^\sigma
\]
and
\[
\val(\Sigma)\cdot f((u,i),(v,j))\in\Nor\Leftrightarrow\val(\Sigma)\cdot f((u,i)^\sigma,(v,j)^\sigma)\in\Nor.
\]
It follows that
\[
(v,j)^\sigma\in (X(u,i))^\sigma\Leftrightarrow(v,j)\in X(u,i)\Leftrightarrow(v,j)^\sigma\in X((u,i)^\sigma).
\]
Thus $(X(u,i))^\sigma=X\left((u,i)^\sigma\right)$ as required.
\end{proof}

\begin{lemma}\label{ijneq}
Let $\Gamma$ and $\Sigma$ be regular graphs with coprime valencies. Suppose that $\Gamma$ is $R$-thin. Then for any $(u, i) \in V(\Gamma \times \Sigma)$ and $(v,j)\in X(u,i)$, we have $i\neq j$.
\end{lemma}

\begin{proof}
Suppose that $i=j$ for some $(u, i) \in V(\Gamma \times \Sigma)$ and $(v,j)\in X(u,i)$. Then $|N_\Sigma(i)\cap N_\Sigma(j)|=\val(\Sigma)$. Since $(v,j)\in X(u,i)$, it follows from Lemma~\ref{vs} that
$$
\frac{|N_\Gamma(u)\cap N_\Gamma(v)|}{\val(\Gamma)}\cdot\val(\Sigma) = \val(\Sigma)\cdot f((u,i),(v,j)) \in\Nor.
$$
This together with $\gcd(\val(\Gamma),\val(\Sigma))=1$ implies that $|N_\Gamma(u)\cap N_\Gamma(v)|=\val(\Gamma)$.
Since $\Gamma$ is $R$-thin, we deduce that $u=v$, and so $(u,i)=(v,j)$, a contradiction.
\end{proof}

\begin{lemma}\label{different}
Let $\Gamma$ and $\Sigma$ be regular graphs with coprime valencies. Suppose that $\Sigma$ is $R$-thin. Then for any $u \in V(\Gamma)$, any $\sigma\in\Aut(\Gamma\times\Sigma)$, and any edge $\{i,j\}$ of $B(\Sigma)$, we have 
$$
(u,i)^{\sigma\pi_\Sigma}\neq(u,j)^{\sigma\pi_\Sigma}.
$$
\end{lemma}

\begin{proof}
Since $\{i,j\}$ is an edge of $B(\Sigma)$, we have $i\neq j$, and hence $(u,i)\neq(u,j)$. Since $\sigma\in\Aut(\Gamma\times\Sigma)$, we then have $(u,i)^\sigma\neq(u,j)^\sigma$. Suppose to the contrary that $(u,i)^{\sigma\pi_\Sigma}=(u,j)^{\sigma\pi_\Sigma}=z$ for some $z\in V(\Sigma)$.
Then $(u,i)^\sigma=(g,z)$ and $(u,j)^\sigma=(h,z)$ for some $g, h \in V(\Gamma)$ with $g\neq h$. Thus
\begin{align*}
|N_{\Gamma\times\Sigma}((u,i)^\sigma)\cap N_{\Gamma\times\Sigma}((u,j)^\sigma)|&=|N_{\Gamma\times\Sigma}((g,z))\cap N_{\Gamma\times\Sigma}((h,z))|\\
&=|N_\Gamma(g)\cap N_\Gamma(h)|\cdot |N_\Sigma(z)|\\
&=|N_\Gamma(g)\cap N_\Gamma(h)|\cdot\val(\Sigma).
\end{align*}
On the other hand, since $\sigma\in\Aut(\Gamma\times\Sigma)$, we have
\begin{align*}
|N_{\Gamma\times\Sigma}((u,i)^\sigma)\cap N_{\Gamma\times\Sigma}((u,j)^\sigma)|&=|N_{\Gamma\times\Sigma}((u,i))\cap N_{\Gamma\times\Sigma}((u,j))|\\
&=|N_\Gamma(u)|\cdot |N_\Sigma(i)\cap N_\Sigma(j)|\\
&=\val(\Gamma)\cdot |N_\Sigma(i)\cap N_\Sigma(j)|.
\end{align*}
Hence
\[
|N_\Gamma(g)\cap N_\Gamma(h)|\cdot\val(\Sigma)=\val(\Gamma)\cdot |N_\Sigma(i)\cap N_\Sigma(j)|.
\]
Since $\gcd(\val(\Gamma),\val(\Sigma))=1$, it follows that $\val(\Sigma)$ divides $|N_\Sigma(i)\cap N_\Sigma(j)|$. Hence $|N_\Sigma(i)\cap N_\Sigma(j)|=\val(\Sigma)$ or $|N_\Sigma(i)\cap N_\Sigma(j)|=0$. However, this is impossible as $\Sigma$ is $R$-thin and $\{i,j\}$ is an edge of $B(\Sigma)$.
This completes the proof.
\end{proof}

\begin{lemma}\label{geqn}
Let $\Gamma$ and $\Sigma$ be regular graphs with coprime valencies. Suppose that both $\Gamma$ and $\Sigma$ are $R$-thin. Then for any edge $\{u,v\}$ of $B(\Gamma)$, any edge $\{i,j\}$ of $B(\Sigma)$, and any element $\sigma\in\Aut(\Gamma\times\Sigma)$, we have
\begin{equation}
\label{eq:geqn}
|N_\Sigma((u,i)^{\sigma\pi_\Sigma})\cap N_\Sigma((v,j)^{\sigma\pi_\Sigma})|\geqslant |N_\Sigma(i)\cap N_\Sigma(j)|.
\end{equation}
\end{lemma}

\begin{proof}
Since $\{u,v\}$ is an edge of $B(\Gamma)$, we have $N_\Gamma(u)\cap N_\Gamma(v)\neq\emptyset$. Consider $w\in N_\Gamma(u)\cap N_\Gamma(v)$. We have
\[
\{w\}\times(N_\Sigma(i)\cap N_\Sigma(j))\subseteq N_{\Gamma\times\Sigma}((u,i))\cap N_{\Gamma\times\Sigma}((v,j)).
\]
Since $\sigma\in\Aut(\Gamma\times\Sigma)$, it follows that
\begin{align*}
(\{w\}\times(N_\Sigma(i)\cap N_\Sigma(j)))^\sigma&\subseteq (N_{\Gamma\times\Sigma}((u,i))\cap N_{\Gamma\times\Sigma}((v,j)))^\sigma\\
&=N_{\Gamma\times\Sigma}((u,i)^\sigma)\cap N_{\Gamma\times\Sigma}((v,j)^\sigma).
\end{align*}
This together with Lemma~\ref{different} implies \eqref{eq:geqn}.
\end{proof}

\begin{lemma}\label{Xk}
Let $\Gamma$ and $\Sigma$ be regular graphs with coprime valencies. Suppose that both $\Gamma$ and $\Sigma$ are $R$-thin. Suppose further that $\Sigma$ is vertex-transitive. Then for any $(u, i) \in V(\Gamma \times \Sigma)$, any $\sigma\in\Aut(\Gamma\times\Sigma)$, and any integer $k \ge 0$, we have 
$$
(X_k(u,i))^\sigma=X_k((u,i)^\sigma).
$$
\end{lemma}

\begin{proof}
Consider an arbitrary vertex $(v,j)$ of $X_k(u,i)$. By Lemma~\ref{X}, 
\begin{equation}\label{inX}
(v,j)^\sigma\in (X_k(u,i))^\sigma\subseteq (X(u,i))^\sigma=X((u,i)^\sigma).
\end{equation}
Since $X_k(u,i)=X(u,i)\cap (V(\Gamma)\times D_k(i))$, we have $j\in D_k(i)$ and $(v,j)\in X(u,i)$, whence $\{u,v\}$ is an edge of $B(\Gamma)$ and $\{i,j\}$ is an edge of $B(\Sigma)$. It then follows from Lemma~\ref{geqn} that
\begin{equation}\label{grater1}
|N_\Sigma((u,i)^{\sigma\pi_\Sigma})\cap N_\Sigma((v,j)^{\sigma\pi_\Sigma})|\geqslant |N_\Sigma(i)\cap N_\Sigma(j)|=n_k.
\end{equation}

First assume that $k=0$. Then $(X_0(u,i))^\sigma=\{(u,i)^\sigma\}=X_0((u,i)^\sigma)$.

Next assume that $k=1$. By \eqref{inX}, we have $(v,j)^\sigma\in X((u,i)^\sigma)$. By Lemma~\ref{ijneq}, we then have
\[
(u,i)^{\sigma\pi_\Sigma}\neq(v,j)^{\sigma\pi_\Sigma},
\]
and so
\[
|N_\Sigma((u,i)^{\sigma\pi_\Sigma})\cap N_\Sigma((v,j)^{\sigma\pi_\Sigma})|\leqslant n_1.
\]
This together with~\eqref{grater1} implies that
\[
|N_\Sigma((u,i)^{\sigma\pi_\Sigma})\cap N_\Sigma((v,j)^{\sigma\pi_\Sigma})|=n_1.
\]
Thus $(v,j)^\sigma\in V(\Gamma)\times D_1((u,i)^{\sigma\pi_\Sigma})$. Combining this with~\eqref{inX}, we then have
\[
(v,j)^\sigma\in X((u,i)^\sigma)\cap \left(V(\Gamma)\times D_1((u,i)^{\sigma\pi_\Sigma})\right)=(X_1(u,i)^\sigma).
\]
Therefore,
\begin{equation}\label{oneside1}
(X_1(u,i))^\sigma\subseteq X_1((u,i)^\sigma).
\end{equation}
It then follows that
\[
(X_1((u,i)^\sigma))^{\sigma^{-1}}\subseteq X_1((u,i)^{\sigma\sigma^{-1}})=X_1(u,i).
\]
Thus
\[
|X_1((u,i)^\sigma)| \leqslant |(X_1(u,i))^\sigma|.
\]
This together with~\eqref{oneside1} implies $(X_1(u,i))^\sigma=X_1((u,i)^\sigma)$ as required.

Now assume that $k\geqslant2$. Suppose by induction that for all integers $\ell$ with $0 \leqslant\ell<k$ we have
\begin{equation}\label{ell}
(X_\ell(u,i))^\sigma=X_\ell((u,i)^\sigma).
\end{equation}
Since $(v,j)\in X_k(u,i)$, we have $(v,j)^\sigma\notin (\cup_{m=0}^{k-1}X_m(u,i))^\sigma$.
It then follows from~\eqref{ell} that
\[
(v,j)^\sigma\notin \cup_{m=0}^{k-1}(X_m(u,i))^\sigma=\cup_{m=0}^{k-1}X_m((u,i)^\sigma).
\]
That is,
\[
(v,j)^\sigma\notin X_m((u,i)^\sigma) \quad\text{for}\quad m=0,1,\dots,k-1.
\]
By the definition of $X_m((u,i)^\sigma$ and~\eqref{inX}, we then obtain that
\[
(v,j)^{\sigma\pi_\Sigma}\notin D_m((u,i)^{\sigma\pi_\Sigma}) \quad\text{for}\quad m=0,1,\dots,k-1.
\]
This together with~\eqref{grater1} implies that $(v,j)^{\sigma\pi_\Sigma}\in D_k((u,i)^{\sigma\pi_\Sigma})$, and hence we drive from~\eqref{inX} that
\[
(v,j)^\sigma\in X((u,i)^\sigma)\cap(V(\Gamma)\times D_k((u,i)^{\sigma\pi_\Sigma}))=X_k((u,i)^\sigma).
\]
Therefore,
$$
(X_k(u,i))^\sigma\subseteq X_k((u,i)^\sigma).
$$
Note that this holds for any $(u, i) \in V(\Gamma \times \Sigma)$ and $\sigma\in\Aut(\Gamma\times\Sigma)$. Replacing $(u,i)$ by $(u,i)^{\sigma}$ and $\sigma$ by $\sigma^{-1}$ in this inclusion, we obtain
$$
(X_k((u,i)^{\sigma}))^{\sigma^{-1}} \subseteq X_k(((u,i)^{\sigma})^{\sigma^{-1}}) = X_k(u,i),
$$
or equivalently,
$$
X_k((u,i)^\sigma) \subseteq (X_k(u,i))^\sigma.
$$
Therefore, $(X_k(u,i))^\sigma=X_k((u,i)^\sigma)$ and the proof is complete by induction.
\end{proof}

\subsection{\texorpdfstring{$Y(u,i)$ and $Y_k(u,i)$}{Y(u,i) and Yk(u,i)}}

\begin{lemma}\label{vs2}
Let $\Gamma$ and $\Sigma$ be regular graphs with coprime valencies. Suppose that $\Sigma$ is $\Gamma$-thin.
Then for any $(u,i)\in V(\Gamma\times\Sigma)$, we have 
$$
Y(u,i)=\{u\}\times N_{B(\Sigma)}(i).
$$
\end{lemma}

\begin{proof}
Let $j\in N_{B(\Sigma)}(i)$. Then $i\neq j$ and $N_\Sigma(i)\cap N_\Sigma(j)\neq\emptyset$. Since $\Gamma$ and $\Sigma$ are regular, we drive from Lemma~\ref{vs} that
\begin{equation}\label{fsu}
\val(\Sigma)\cdot f((u,i),(u,j))=|N_\Sigma(i)\cap N_\Sigma(j)|\in\Nor.
\end{equation}
For any $w\in V(\Gamma)$, by Lemma~\ref{vs} and~\eqref{fsu} we have
\begin{align} 
\nonumber
\val(\Sigma)\cdot f((u,i),(w,j))&=\frac{|N_\Gamma(u)\cap N_\Gamma(w)|}{\val(\Gamma)}\cdot|N_\Sigma(i)\cap N_\Sigma(j)|\\ \nonumber
&\leqslant |N_\Sigma(i)\cap N_\Sigma(j)|\\ \nonumber
&=\val(\Sigma)\cdot f((u,i),(u,j)).\nonumber
\end{align}
Hence
\begin{equation}\label{uuw}
f((u,i),(u,j))\geqslant f((u,i),(w,j)) \ \text{ for all }\ w\in V(\Gamma).
\end{equation}
This together with~\eqref{fsu} implies that $(u,j)\in Y(u,i)$. Thus
\begin{equation}\label{l}
\{u\}\times N_{B(\Sigma)}(i)\subseteq Y(u,i).
\end{equation}

Conversely, let $(v,j)\in Y(u,i)$. Then $(v,j)\neq (u,i)$, $\val(\Sigma)\cdot f((u,i),(v,j))\in\Nor$, and
\begin{equation}\label{max2}
f((u,i),(v,j))\geqslant f((u,i),(w,j)) \text{ for all } w\in V(\Gamma).
\end{equation}
Moreover, by Lemma~\ref{vs}, we have
\begin{equation}\label{fsij}
\val(\Sigma)\cdot f((u,i),(v,j))=\frac{|N_\Gamma(u)\cap N_\Gamma(v)|}{\val(\Gamma)}\cdot|N_\Sigma(i)\cap N_\Sigma(j)|\in\Nor.
\end{equation}
It follows that $|N_\Sigma(i)\cap N_\Sigma(j)|\neq 0$, and so $N_\Sigma(i)\cap N_\Sigma(j) \neq \emptyset$. Since $(v,j)\in Y(u,i)\subseteq X(u,i)$, we derive from Lemma~\ref{ijneq} that $j\neq i$, and hence $j\in N_{B(\Sigma)}(i)$. Thus, using the same argument as in the first paragraph of this proof, we can derive that $(u, i), (u, j)$ and $(w, j)$ satisfy \eqref{uuw}. Combining this and \eqref{max2}, we then obtain $v=u$ and therefore $(v,j)\in \{u\}\times N_{B(\Sigma)}(i)$. So we have proved that
$$
Y(u,i)\subseteq\{u\}\times N_{B(\Sigma)}(i),
$$
which together with \eqref{l} yields $Y(u,i)=\{u\}\times N_{B(\Sigma)}(i)$, as required.
\end{proof}

\begin{lemma}\label{XY}
Let $\Gamma$ and $\Sigma$ be regular graphs with coprime valencies. Suppose that both $\Gamma$ and $\Sigma$ are $R$-thin. Then the following hold for any $(u,i)\in V(\Gamma\times\Sigma)$ and positive integer $k$:
\begin{enumerate}[{\rm (a)}]
\item $Y_k(u,i)\subseteq X_k(u,i)$;
\item $Y_k(u,i)=Y(u,i)\cap(V(\Gamma)\times D_k(i))$.
\end{enumerate}
\end{lemma}

\begin{proof}
If $Y_k(u,i)=\emptyset$, then statement~(a) is obvious. Now assume that  $Y_k(u,i)\neq\emptyset$ and let $(v,j)\in Y_k(u,i)$. Then
\begin{equation}\label{ya1}
(v,j)\in X(u,i)\setminus (\cup_{m=0}^{k-1}X_m(u,i))
\end{equation}
and
\begin{equation}\label{ya3}
f((u,i),(v,j))\geqslant f((u,i),(w,k)) \text{ for all $(w,k)\in X(u,i)\setminus (\cup_{m=0}^{k-1}X_m(u,i))
$}.
\end{equation}
Suppose by way of contradiction that $(v,j)\notin X_k(u,i)$.
Then by \eqref{ya1} we have
\[
(v,j)\in X(u,i) \text{ but } (v,j)\notin X_m(u,i) \text{ for } m=0,1,\dots,k-1,k.
\]
Since $X_m(u,i)=X(u,i)\cap (V(\Gamma)\times D_m(i))$, we have
\[
(v,j)\notin V(\Gamma)\times D_m(i)\quad\text{for}\quad m=0,1,\dots,k-1,k.
\]
So $j\in D_t(i)$ for some $t>k$. Thus $D_t(i)\neq\emptyset$ and so $D_k(i)\neq\emptyset$. Take $y\in D_k(i)$.
Then
\begin{equation}\label{le}
|N_\Sigma(i)\cap N_\Sigma(j)|<|N_\Sigma(i)\cap N_\Sigma(y)|
\end{equation}
and
\[
(u,y)\in X(u,i)\setminus (\cup_{m=0}^{k-1}X_m(u,i)).
\]
Using \eqref{le}, we can easily obtain $f((u,i),(v,j)) < f((u,i),(u,y))$, 
which contradicts \eqref{ya3}. This contradiction shows that $(v,j)\in X_k(u,i)$, and therefore $Y_k(u,i)\subseteq X_k(u,i)$. This completes the proof of statement~(a).

Now we prove statement~(b). Let $(v,j)\in Y_k(u,i)$. By statement~(a), we have
\begin{equation}\label{in1}
(v,j)\in X_k(u,i)=X(u,i)\cap(V(\Gamma)\times D_k(i))\subseteq X(u,i).
\end{equation}
It then follows from Lemmas~\ref{vs} and \ref{ijneq} that
$$
\val(\Sigma)\cdot f((u,i),(v,j))=\frac{|N_\Gamma(u)\cap N_\Gamma(v)|}{\val(\Gamma)}\cdot |N_\Sigma(i)\cap N_\Sigma(j)|\in\Nor
$$
and $i\neq j$.
Hence $N_\Sigma(i)\cap N_\Sigma(j) \ne \emptyset$. Using Lemma~\ref{vs}, it is straightforward to verify that
\begin{equation}\label{uN}
\val(\Sigma)\cdot f((u,i),(u,j))\in\Nor
\end{equation}
and
\begin{equation}\label{ug}
\val(\Sigma)\cdot f((u,i),(u,j))\geqslant\val(\Sigma)\cdot f((u,i),(v,j)).
\end{equation}
Moreover, by~\eqref{in1}, we have $j\in D_k(i)$, and so $j\notin \cup_{m=0}^{k-1}D_m(i)$. Since $i\neq j$, we derive from~\eqref{uN} that
\[
(u,j)\in X(u,i)\setminus (\cup_{m=0}^{k-1}X_m(u,i)).
\]
This combined with~\eqref{ya3} and~\eqref{ug} implies that $v=u$.
Since $N_\Sigma(i)\cap N_\Sigma(j) \ne \emptyset$, we have $j\in N_{B(\Sigma)}(i)$, and thus we obtain from Lemma~\ref{vs2} and~\eqref{in1} that $(v,j)\in Y(u,i)\cap V(\Gamma)\times D_k(i)$. Therefore,
\begin{equation}\label{leftside}
Y_k(u,i)\subseteq Y(u,i)\cap (V(\Gamma)\times D_k(i)).
\end{equation}

Conversely, let $(v,j)\in Y(u,i)\cap(V(\Gamma)\times D_k(i))$. Since $Y(u,i)\subseteq X(u,i)$, we have
\begin{equation}\label{y1}
(v,j)\in X(u,i)\cap(V(\Gamma)\times D_k(i))\subseteq X(u,i)\setminus (\cup_{m=0}^{k-1}X_m(u,i)). \end{equation}
Consider an arbitrary element $(w,t)$ of $X(u,i)\setminus (\cup_{m=0}^{k-1}X_m(u,i))$. Since
$X_m(u,i)=X(u,i)\cap (V(\Gamma)\times D_m(i))$,
we have
\[
(w,t)\notin X(u,i)\cap (V(\Gamma)\times D_m(i))\quad\text{for}\quad m=0,1,\dots,k-1.
\]
Since $(w,t)\in X(u,i)$, we deduce that
\[
t\notin D_m(i)\quad\text{for}\quad m=0,1,\dots,k-1,
\]
and hence $j\notin \cup_{m=0}^{k-1}D_m(i)$.
Note that $(v,j)\in D_k(i)$.
It follows that
\begin{equation}\label{tj}
|N_\Sigma(i)\cap N_\Sigma(j)|\geqslant |N_\Sigma(i)\cap N_\Sigma(t)|.
\end{equation}
Since $(v,j)\in Y(u,i)$, by Lemma~\ref{vs2} we obtain that $u=v$. Thus
\[
|N_\Gamma(u)\cap N_\Gamma(v)|=\val(\Gamma)\geqslant |N_\Gamma(u)\cap N_\Gamma(w)|.
\]
This together with~\eqref{tj} and Lemma~\ref{vs} implies that
\begin{align*}
\val(\Sigma)\cdot f((u,i),(v,j))&=|N_\Sigma(i)\cap N_\Sigma(j)|\\
&\geqslant\frac{|N_\Gamma(u)\cap N_\Gamma(w)|}{\val(\Gamma)}\cdot|N_\Sigma(i)\cap N_\Sigma(t)|\\
&=\val(\Sigma)\cdot f((u,i),(w,t)).
\end{align*}
Combining this with~\eqref{y1}, we obtain $(v,j)\in Y_k(u,i)$. Thus
\[
Y(u,i)\cap(V(\Gamma)\times D_k(i))\subseteq Y_k(u,i).
\]
This together with \eqref{leftside} completes the proof of statement~(b).
\end{proof}



\begin{lemma}\label{hom}
Let $\Gamma$ and $\Sigma$ be regular graphs with coprime valencies. Suppose that both $\Gamma$ and $\Sigma$ are $R$-thin. Suppose further that $\Sigma$ is vertex-transitive. Then for any $(u,i)\in V(\Gamma\times\Sigma)$ and $\sigma\in\Aut(\Gamma\times\Sigma)$, we have 
$$
(Y(u,i))^\sigma=Y((u,i)^\sigma).
$$
\end{lemma}

\begin{proof}
Since $\Sigma$ is finite, there exists a positive integer $t$ such that $V(\Sigma)=\cup_{m=0}^{t+1}D_m(i)$,
where as before $D_0(i)=\{i\}$ and $D_{t+1}(i)=\{x\in V(\Sigma)\mid N_\Sigma(i)\cap N_\Sigma(x)=\emptyset\}$.
Since $Y(u,i)\subseteq X(u,i)$, by Lemma~\ref{ijneq}, we have
\begin{equation}\label{00}
Y(u,i)\cap (V(\Gamma)\times D_0(i))=\emptyset.
\end{equation}
For any $(v,j)\in V(\Gamma)\times D_{t+1}(i)$, we have $j\notin N_{B(\Sigma)}(i)$ as $N_\Sigma(i)\cap N_\Sigma(j)=\emptyset$. By Lemma~\ref{vs2}, we then have $(v,j)\notin Y(u,i)$ and therefore
\begin{equation}\label{11}
Y(u,i)\cap (V(\Gamma)\times D_{t+1}(i))=\emptyset.
\end{equation}
Note that
\begin{align*}
V(\Gamma\times\Sigma) & = V(\Gamma)\times V(\Sigma)\\
& = V(\Gamma)\times \left(\cup_{m=0}^{t+1}D_m(i)\right)\\
& = \cup_{m=0}^{t+1}(V(\Gamma)\times D_m(i))\\
& = \left(\cup_{m=1}^{t}(V(\Gamma)\times D_m(i))\right)\cup \left(V(\Gamma)\times D_0(i)\right) \cup \left(V(\Gamma)\times D_{t+1}(i)\right).
\end{align*}
Thus, by \eqref{00}, \eqref{11} and Lemma~\ref{XY}(b), we have 
\begin{align*}
Y(u,i)&=Y(u,i)\cap V(\Gamma\times\Sigma)\\
&=Y(u,i)\cap\left(\cup_{m=1}^{t}(V(\Gamma)\times D_m(i))\right)\\
&=\cup_{m=1}^{t}(Y(u,i)\cap(V(\Gamma)\times D_m(i))\\
&=\cup_{m=1}^t Y_m(u,i).
\end{align*}

Since $\sigma\in\Aut(\Gamma\times\Sigma)$, by Lemmas~\ref{X} and \ref{Xk}, we obtain that for each positive integer $k$,
\begin{align*}
\left(X(u,i)\setminus \left(\cup_{m=0}^{k-1}X_m(u,i)\right)\right)^\sigma & = (X(u,i))^\sigma \setminus \left(\cup_{m=0}^{k-1}(X_m(u,i))^\sigma\right)\\
& = X((u,i)^\sigma)\setminus\left(\cup_{m=0}^{k-1}X_m((u,i)^\sigma)\right).
\end{align*}
Thus $(Y_k(u,i))^\sigma=Y_k((u,i)^\sigma)$ by the definition of $Y_k(u,i)$.
Therefore, 
\[
(Y(u,i))^\sigma=\left(\cup_{m=1}^tY_m(u,i)\right)^\sigma = \cup_{m=1}^t (Y_m(u,i))^\sigma = \cup_{m=1}^tY_m((u,i)^\sigma)=Y((u,i)^\sigma),
\]
completing the proof.
\end{proof}

\section{Proof of Theorem~\ref{main-theorem2}}\label{ProofTh2}

\subsection{Lemmas}

We need the following lemmas in our proof of Theorem~\ref{main-theorem2}.

\begin{lemma}\label{add2}
Let $\Gamma$ and $\Sigma$ be regular graphs with coprime valencies. Suppose that both $\Gamma$ and $\Sigma$ are $R$-thin. Suppose further that $\Sigma$ is vertex-transitive. Let $u\in V(\Gamma)$ and $i,j\in V(\Sigma)$. If $N_\Sigma(i)\cap N_\Sigma(j)\neq\emptyset$, then for any $\sigma\in\Aut(\Gamma\times\Sigma)$, we have 
$$
(u,i)^{\sigma\pi_\Gamma}=(u,j)^{\sigma\pi_\Gamma}.
$$
\end{lemma}

\begin{proof}
If $i=j$, then $(u,i)^\sigma=(u,j)^\sigma$ and so the statement is true. Now assume that $i\neq j$ and $N_\Sigma(i)\cap N_\Sigma(j)\neq\emptyset$. Then $j\in N_{B(\Sigma)}(i)$. By Lemma~\ref{vs2}, we then have $(u,j)\in Y(u,i)$. Hence 
$(u,j)^\sigma\in (Y(u,i))^\sigma=Y((u,i)^\sigma)$ by Lemma~\ref{hom}.
This together with Lemma~\ref{vs2} implies that $(u,i)^{\sigma\pi_\Gamma}=(u,j)^{\sigma\pi_\Gamma}$, as required.
\end{proof}

\begin{lemma} 
\label{even}
Let $\Gamma$ and $\Sigma$ be regular graphs with coprime valencies. Suppose that both $\Gamma$ and $\Sigma$ are $R$-thin. Suppose further that $\Sigma$ is vertex-transitive. Let $u,v\in V(\Gamma)$ and $i\in V(\Sigma)$. If there is a walk from $u$ to $v$ of even length, then for any $\sigma\in\Aut(\Gamma\times\Sigma)$, we have 
\begin{equation}
\label{eq:even}
(u,i)^{\sigma\pi_\Sigma}=(v,i)^{\sigma\pi_\Sigma}.
\end{equation}
\end{lemma}

\begin{proof}
It suffices to prove \eqref{eq:even} in the case when $N_\Gamma(u)\cap N_\Gamma(v)\neq\emptyset$. The result in the general case follows by applying this equation to pairs of every other vertices on the walk.  

So let us assume that $N_\Gamma(u)\cap N_\Gamma(v)\neq\emptyset$. Take a vertex $w\in N_\Gamma(u)\cap N_\Gamma(v)$. Set $W=\{w\}\times N_\Sigma(i)$. Then $W\subseteq N_{\Gamma\times\Sigma}((u,i))\cap N_{\Gamma\times\Sigma}((v,i))$.
Since $\sigma\in\Aut(\Gamma\times\Sigma)$, it follows that
\[
W^\sigma\subseteq N_{\Gamma\times\Sigma}((u,i)^\sigma)\cap N_{\Gamma\times\Sigma}((v,i)^\sigma).
\]
Consequently, 
\begin{equation}\label{nei}
W^{\sigma\pi_\Sigma}\subseteq N_{\Sigma}((u,i)^{\sigma\pi_\Sigma})\cap N_{\Sigma}((v,i)^{\sigma\pi_\Sigma}).
\end{equation}
Note that
\begin{equation}\label{val}
|W^\sigma|=|W|=|N_\Sigma(i)|=|N_{\Sigma}((u,i)^{\sigma\pi_\Sigma})|=|N_{\Sigma}((v,i)^{\sigma\pi_\Sigma})|=\val(\Sigma).
\end{equation}
Consider any two distinct vertices $(w,j), (w,k)$ in $W$. We have $j\neq k\in N_\Sigma(i)$ and $i\in N_\Sigma(j)\cap N_\Sigma(k)$, whence $N_\Sigma(j)\cap N_\Sigma(k)\neq\emptyset$. By Lemma~\ref{add2}, we have $(w,j)^{\sigma\pi_\Gamma}=(w,k)^{\sigma\pi_\Gamma}$. Since $(w,j)\neq(w,k)$ and $\sigma\in\Aut(\Gamma\times\Sigma)$, we then obtain that $(w,j)^{\sigma\pi_\Sigma}\neq(w,k)^{\sigma\pi_\Sigma}$. Hence $|W^{\sigma\pi_\Sigma}|=|W^\sigma|$. Combining this with \eqref{nei} and \eqref{val}, we obtain that
\[
N_{\Sigma}((u,i)^{\sigma\pi_\Sigma})=N_{\Sigma}((v,i)^{\sigma\pi_\Sigma})=W^{\sigma\pi_\Sigma}.
\]
Since $\Sigma$ is $R$-thin, we then conclude that $(u,i)^{\sigma\pi_\Sigma}=(v,i)^{\sigma\pi_\Sigma}$.
\end{proof}
 
For a graph $\Delta$ and a partition $\calB$ of $V(\Delta)$, the \emph{quotient graph} $\Delta_\calB$ of $\Delta$ with respect to $\calB$ is defined to have vertex set $\calB$ such that two blocks $B, C \in \calB$ are adjacent if and only if there exists at least one edge of $\Delta$ with one end-vertex in $B$ and the other end-vertex in $C$. 

\begin{lemma}\label{Gamma-bipartite}
Let $\Gamma$ and $\Sigma$ be connected regular graphs with coprime valencies. Suppose that both $\Gamma$ and $\Sigma$ are $R$-thin. Suppose further that $\Gamma$ is bipartite and $\Sigma$ is vertex-transitive and non-bipartite. Then 
$$
\Aut(\Gamma\times\Sigma)=P(\Gamma,\Sigma).
$$
\end{lemma}

\begin{proof}
Let $B_0$ and $B_1$ be the biparts of $\Gamma$, and let $V(\Sigma)=\{1,2,\dots,n\}$. Let $u,v\in V(\Gamma)$ and $i\in V(\Sigma)$, and let $\sigma$ be an arbitrary element of $\Aut(\Gamma\times\Sigma)$. Let 
$$
\Delta = \Gamma\times\Sigma.
$$ 
Set
\[
\calB=\{V(\Gamma)\times\{j\}\mid j\in V(\Sigma)\}
\]
and 
\[
\calD=\{B_0\times\{j\}, B_1\times\{j\}\mid j\in V(\Sigma)\}. 
\]
Then $\calB$ and $\calD$ are partitions of $V(\Delta)$, and hence $\Delta_{\calB}$ and $\Delta_{\calD}$ are well defined.

First assume that $u$ and $v$ are in the same biparts of $\Gamma$. Since $\Gamma$ is connected and bipartite, there is a walk in $\Gamma$ from $u$ to $v$ of even length. It then follows from Lemma \ref{even} that $(u,i)^{\sigma\pi_\Sigma}=(v,i)^{\sigma\pi_\Sigma}$. Hence $\sigma$ preserves $\calD$, and so $\sigma\in\Aut(\Delta_\calD)$.

Next assume that $u$ and $v$ are in different biparts of $\Gamma$. Without loss generality we may assume that $u\in B_0$ and $v\in B_1$. Since $\Sigma$ is non-bipartite and vertex-transitive, every vertex of $\Sigma$ is contained in an odd cycle, and so there exists a cycle of odd length containing $i$, say, $C: i_1,i_2,\dots,i_\ell,i_1$, where $i=i_1$ and $\ell$ is odd. It follows that
\[
C_\calB: V(\Gamma)\times\{i_1\},V(\Gamma)\times\{i_2\},\dots,V(\Gamma)\times\{i_\ell\}, V(\Gamma)\times\{i_1\}
\]
and
\[
C_\calD: B_0\times\{i_1\},B_1\times\{i_2\},\dots,B_0\times\{i_\ell\}, B_1\times\{i_1\},B_0\times\{i_2\},\dots,B_1\times\{i_\ell\}, B_0\times\{i_1\}
\]
are cycles in $\Delta_\calB$ and $\Delta_\calD$, respectively.
Note that $C_\calD$ is of length $2\ell$. Since $\sigma\in\Aut(\Delta_\calD)$ as shown above, 
$\sigma$ maps $C_\calD$ to the cycle $(C_\calD)^\sigma$ of $\Delta_\calD$, and also maps the pair of antipodals $B_0\times\{i\}, B_1\times\{i\}$ to some pair of antipodals in $(C_\calD)^\sigma$.
Note also that for each pair of antipodals $X, Y$ of some cycle in $\Delta_\calD$, we have $X^{\pi_\Sigma}=Y^{\pi_\Sigma}$ and $|X^{\pi_\Sigma}|=|Y^{\pi_\Sigma}|=1$. In particular, we have
\[
(B_0\times\{i\})^{\sigma\pi_\Sigma}=(B_1\times\{i\})^{\sigma\pi_\Sigma}
\]
and
 \[
 \quad |(B_0\times\{i\})^{\sigma\pi_\Sigma}|=|(B_1\times\{i\})^{\sigma\pi_\Sigma} |=1.
\]
Since $u\in B_0$ and $v\in B_1$, we then obtain that $(u,i)^{\sigma\pi_\Sigma}=(v,i)^{\sigma\pi_\Sigma}$. Since this holds for any $u,v\in V(\Gamma)$ and $i\in V(\Sigma)$, we conclude that $\sigma\in P(\Gamma,\Sigma)$. Since this holds for any $\sigma \in \Aut(\Gamma\times\Sigma)$, we obtain that $\Aut(\Gamma\times\Sigma)\leqslant P(\Gamma,\Sigma)$, which yields $\Aut(\Gamma\times\Sigma)=P(\Gamma,\Sigma)$.
\end{proof}

\begin{lemma}
\label{Sigma-bipartite}
Let $\Gamma$ and $\Sigma$ be connected regular graphs with coprime valencies. Suppose that both $\Gamma$ and $\Sigma$ are $R$-thin. Suppose further that $\Gamma$ is non-bipartite and $\Sigma$ is vertex-transitive and bipartite. Then 
$$
\Aut(\Gamma\times\Sigma)=P(\Gamma,\Sigma).
$$
\end{lemma}

\begin{proof}
Let 
$$
\Delta = \Gamma\times\Sigma
$$ 
and 
$$
\calB=\{V(\Gamma)\times\{j\}\mid j\in V(\Sigma)\}.
$$ 
Then the quotient graph $\Delta_{\calB}$ of $\Delta$ with respect to $\calB$ is isomorphic to $\Sigma$. Since $\Sigma$ is bipartite, it follows that there is no odd cycle in $\Delta_{\calB}$. This implies that, for any $u,v\in V(\Gamma)$ and $i\in V(\Sigma)$, there is no path of odd length from $(u,i)$ to $(v,i)$ in $\Delta$. Since $\Gamma$ is non-bipartite, we obtain from Lemma \ref{cvn}(a) that $\Delta$ is connected. It follows that there is a path of even length from $(u,i)$ to $(v,i)$ in $\Delta$, and therefore there is a walk of even length from $u$ to $v$ in $\Gamma$. Thus, by Lemma \ref{even}, for any $\sigma \in \Aut(\Gamma\times\Sigma)$, we have $(u,i)^{\sigma\pi_\Sigma}=(v,i)^{\sigma\pi_\Sigma}$, which implies $\sigma\in P(\Gamma,\Sigma)$. Since this holds for any $\sigma \in \Aut(\Gamma\times\Sigma)$, it follows that $\Aut(\Gamma\times\Sigma)\leqslant P(\Gamma,\Sigma)$, yielding $\Aut(\Gamma\times\Sigma)=P(\Gamma,\Sigma)$ as required. 
\end{proof}
 
\subsection{Proof of Theorem~\ref{main-theorem2}}

We are now ready to prove Theorem~\ref{main-theorem2}.

\begin{proof}
Let $\Gamma$ be a regular graph and $\Sigma$ a vertex-transitive graph such that $\val(\Gamma)$ and $\val(\Sigma)$ are coprime.  

Suppose that $(\Gamma,\Sigma)$ is nontrivially unstable. Then by the definition of a nontrivially unstable graph pair, $\Gamma$ and $\Sigma$ are coprime connected $R$-thin graphs and at least one of them is non-bipartite. Hence, by Lemma \ref{cvn}, $\Gamma\times\Sigma$ is connected and $R$-thin. Moreover, by Lemmas~\ref{Gamma-bipartite} and \ref{Sigma-bipartite}, we have $\Aut(\Gamma\times\Sigma)=P(\Gamma,\Sigma)$. 
Since $(\Gamma,\Sigma)$ is unstable, it follows from Lemma \ref{nontrivial}(b) that at least one $\Sigma$-automorphism of $\Gamma$ is nondiagonal. 

Conversely, suppose that $\Gamma\times\Sigma$ is connected and $R$-thin and at least one $\Sigma$-automorphism of $\Gamma$ is nondiagonal. Then, by Lemma \ref{cvn}, $\Gamma$ and $\Sigma$ are connected $R$-thin graphs and at least one of them is non-bipartite. If both $\Gamma$ and $\Sigma$ are non-bipartite, then by Lemma \ref{coprime}(b), $(\Gamma,\Sigma)$ is stable, but this contradicts Lemma \ref{nontrivial}(a) as we assume that at least one $\Sigma$-automorphism of $\Gamma$ is nondiagonal. Thus exactly one of $\Gamma$ and $\Sigma$ is non-bipartite. Since $\val(\Gamma)$ and $\val(\Sigma)$ are coprime, by Remark~\ref{coprime-valencies}, $\Gamma$ and $\Sigma$ are coprime. Finally, by Lemmas \ref{Gamma-bipartite} and \ref{Sigma-bipartite}, we have $\Aut(\Gamma\times\Sigma)=P(\Gamma,\Sigma)$. Hence, by Lemma \ref{nontrivial}(b), $(\Gamma,\Sigma)$ is unstable. Moreover, $(\Gamma,\Sigma)$ is nontrivially unstable since $\Gamma$ and $\Sigma$ are coprime connected $R$-thin graphs.
\end{proof}

\section{Concluding remarks}
\label{Open}
 
In the case when $\Sigma = K_2$, Theorem \ref{main-theorem2} gives rise to the following result: A connected regular graph is unstable if and only if it has a nontrivial two-fold automorphism. As mentioned in the introduction, this is a special case of \cite[Theorem~3.2]{LMS2015}, where the graph in the statement is not required to be regular. So Theorem \ref{main-theorem2} is a partial generalization of \cite[Theorem~3.2]{LMS2015}. It would be interesting to study whether the result in Theorem \ref{main-theorem2} is still true if $\Sigma$ is not required to be regular with valency coprime to the valency of $\Gamma$.

As seen in Theorem \ref{main-theorem2}, there are close connections between stability of graph pairs $(\Gamma, \Sigma)$ and $\Sigma$-automorphisms of $\Gamma$. However, the notion of $\Sigma$-automorphisms of $\Gamma$ deserves further studies for its own sake. As far as we know, there are not many results on $\Sigma$-automorphisms of $\Gamma$ in the literature, except in the case when $\Sigma = K_1^{\circ}$ for which a $K_1^{\circ}$-automorphism of $\Gamma$ is an automorphism of $\Gamma$ in the usual sense, and in the case when $\Sigma = K_2$ for which a $K_2$-automorphism of $\Gamma$ is exactly a two-fold automorphism of $\Gamma$ \cite{LMS2015}. The concept of two-fold automorphisms was first introduced by Zelinka in \cite{Zelinka71, Zelinka72} for digraphs in his study of isotopies of digraphs and was extended to mixed graphs by Lauri et al. in \cite{LMS2015}. It is readily seen that, if $\Sigma_1$ is a spanning subgraph of $\Sigma_2$, then $\Aut_{\Sigma_2}(\Gamma) \le \Aut_{\Sigma_1}(\Gamma)$. So among all graphs $\Sigma$ of order $n$ the complete graph $K_n$ gives rise to the smallest possible group $\Aut_{\Sigma}(\Gamma)$. Therefore, it would be interesting to study the groups $\Aut_{K_n}(\Gamma)$ for $n \ge 2$. Note that $\Aut_{K_2}(\Gamma)$ is exactly the two-fold automorphism group $\Aut^{\mathrm{TF}}(\Gamma)$ of $\Gamma$ \cite{LMS2015}.

Finally, many questions about (nontrivially) unstable graph pairs may be asked. For example, by Lemma \ref{nontrivial}(a), if $\Gamma$ admits a nondiagonal $\Sigma$-automorphism then $(\Gamma, \Sigma)$ is unstable, and Theorem \ref{main-theorem2} determines a situation where this necessary condition is also sufficient. In general, one may ask the following question: For an unstable pair of graphs $(\Gamma,\Sigma)$,  under what conditions does $\Gamma$ admit a nondiagonal $\Sigma$-automorphism? One may also study the stability of $(\Gamma, \Sigma)$ for various special families of graphs $\Gamma$ and/or various special families of graphs $\Sigma$. Possible candidates for $\Gamma$ include circulant graphs \cite{QXZ2019, Wilson2008}, arc-transitive graphs, generalized Petersen graphs $\mathrm{GP}(n,k)$ \cite{QXZ2020}, etc. and potential choices for $\Sigma$ include complete graphs $K_n$, complete bipartite graphs $K_{n,n}$, cycles $C_n$, etc. Obviously, this area of research is wide open and pleasantly inviting.
 
\vskip0.1in
\noindent\textsc{Acknowledgements.}  
We are grateful to the anonymous referees for their helpful comments and Dr Jiyong Chen for his valuable advices. Part of the work was done during a visit of the first author to The University of Melbourne. The first author would like to thank The University of Melbourne for its hospitality during her visit and Beijing Jiaotong University for its financial support. She also thanks the National Natural Science Foundation of China (11671030) for its financial support during her PhD. The first author was supported by the Fundamental Research Funds for Beijing Universities allocated to Capital University of Economics and Business(XRZ2020058). The third author was supported by the National Natural Science Foundation of China (11671030,12071023).


\begin{thebibliography}{}


\bibitem{BCP1997}
W. Bosma, J. Cannon and C. Playoust,
The Magma Algebra System I: The User Language,
\emph{J. Symbolic Comput.} 24 (1997), 235--265.

\bibitem{Chen}
J. Chen,
Embeddings of direct product graphs, in preparation, personal communication.

\bibitem{Dorfler1974}
W. D\"{o}rfler,
Primfaktorzerlegung und Automorphismen des Kardinalproduktes von Graphen,
\emph{Glasnik Mat. Ser. III} 9(29) (1974), 15--27.


\bibitem{HIK2011}
R. Hammack, W. Imrich and S. Klav\v{z}ar,
\emph{Handbook of Product Graphs}, 2nd ed., CRC Press 2011.

\bibitem{LMS2015}
J. Lauri, R. Mizzi and R. Scapellato,
Unstable graphs: a fresh outlook via TF-automorphisms,
\emph{Ars Math. Contemp.}  8 (2015),  115--131.

\bibitem{MSZ1989}
D. Maru\v{s}i\v{c}, R. Scapellato and N. Zagaglia Salvi,
A characterization of particular symmetric $(0,1)$ matrices,
\emph{Linear Algebra Appl.} 119 (1989), 153--162.

\bibitem{MSZ1992}
D. Maru\v{s}i\v{c}, R. Scapellato, and N. Zagaglia Salvi,
Generalized Cayley graphs,
\emph{Discrete Math.} 102 (1992), 279--285.

\bibitem{NS1996}
R. Nedela and M. \v Skoviera,
Regular embeddings of canonical double coverings of graphs,
\emph{J. Combin. Theory Ser. B} 67 (1996), 249--277.


\bibitem{QXZ2019}
Y-L. Qin, B. Xia and S. Zhou,
Stability of circulant graphs,
\emph{J. Combin. Theory Ser. B}, 136 (2019), 154--169.

\bibitem{QXZ2020}
Y-L. Qin, B. Xia and S. Zhou,
Stability of generalized Petersen graphs, \emph{J. Graph Theory}, accepted, DOI: 10.1002/jgt.22642. 

\bibitem{Surowski2001}
D. Surowski,
Stability of arc-transitive graphs,
\emph{J. Graph Theory} 38  (2001), 95--110.

\bibitem{Surowski2003}
D. Surowski,
Automorphism groups of certain unstable graphs,
\emph{Math. Slovaca}  53  (2003), 215--232.

\bibitem{Weichsel62}
P. M. Weichsel, The Kronecker product of graphs, 
\emph{Proc. Amer. Math. Soc.} 13 (1962), 47--52.

\bibitem{Wilson2008}
S. Wilson,
Unexpected symmetries in unstable graphs,
\emph{J. Combin. Theory Ser. B}  98  (2008), 359--383.

\bibitem{Zelinka71}
B. Zelinka, The group of autotopies of a digraph, \emph{Czech Math. J.} 21
(1971), 619--624.

\bibitem{Zelinka72}
B. Zelinka, Isotopy of digraphs, \emph{Czech Math. J.} 22 (1972), 353--360.

\bibliographystyle{100}

\end{thebibliography}
\end{document}